\documentclass[12pt]{amsart}
\usepackage{epsfig,color,mathtools}

\makeatother

\theoremstyle{definition}

\def\fnum{equation} 
\newtheorem{Thm}[\fnum]{Theorem}
\newtheorem{Cor}[\fnum]{Corollary}

\newtheorem{Lem}[\fnum]{Lemma}

\newtheorem{Def}[\fnum]{Definition}

\newtheorem{Pro}[\fnum]{Proposition}

\numberwithin{equation}{section}

\newcommand{\nn}{{\bf{n}}}

\newcommand{\dist}{{\text {dist}}}

\newcommand{\Hess}{{\text {Hess}}}

\def\RR{{\mathbb R}}

\def\SS{{\bold S}}

\newcommand{\e}{{\text {e}}}

\newcommand{\cA}{{\mathcal{A}}}

\newcommand{\cD}{{\mathcal{D}}}

\newcommand{\eqr}[1]{(\ref{#1})}

\title{Distance between minimal surfaces and  flows}

\author{Tobias Holck Colding}%
\address{MIT, Dept. of Math.\\
77 Massachusetts Avenue, Cambridge, MA 02139-4307.}
\author{William P. Minicozzi II}%

\thanks{The  authors
were partially supported by NSF  DMS Grants   2405393 and 2304684.}

\email{colding@math.mit.edu and minicozz@math.mit.edu}

\begin{document}

\maketitle

\begin{abstract}
We will show that the distance between two minimal hypersurfaces is a Lipschitz continuous supersolution, in the viscosity sense, of a natural elliptic partial differential equation. This not only recovers several well-known properties of minimal hypersurfaces, but also encodes substantially richer information.

Moreover, if the reference hypersurface is allowed to evolve by mean curvature flow, one obtains comparably strong estimates for a corresponding parabolic PDE, leading in particular to local Harnack inequalities for the distance. There is even a fully parabolic extension in which both hypersurfaces evolve. The problem of tracking the distance between two evolving hypersurfaces arises naturally in a wide range of settings.
\end{abstract}


\section{Introduction}

Let $\Sigma \subset \mathbb{R}^{n+1}$ be a proper\footnote{An immersion is proper if its intersection with every compact set is compact.}, connected, immersed hypersurface, and define the distance function to $\Sigma$ by
\begin{align}\label{e:defd}
d(x) = \inf_{y\in \Sigma} |x - y| \,  .
\end{align}
If $\Sigma$ is a hyperplane, then $d$ coincides with a coordinate function. In general, however, $d$ is significantly more intricate and need not be differentiable everywhere. Nevertheless, $d$ is nonnegative and Lipschitz continuous with Lipschitz constant one. Consequently, by Rademacher's theorem, $d$ is differentiable  almost everywhere in $\RR^{n+1}$.

The geometry of $\Sigma$ is captured in the second derivatives of $d$, while the restriction of $d$ to a second hypersurface
reflects the geometry of both.
Suppose that $\Sigma$ is minimal, $\Gamma_t$ is a mean curvature flow (MCF), and  $\Gamma_t$ is disjoint from $\Sigma$. Let
 $\Box_{\Gamma_t} =  \partial_t - \Delta_{\Gamma_t}$ be the heat operator on $\Gamma_t$. 
The function $d$ satisfies a natural differential inequality on $\Gamma_t$ in the viscosity sense:

\begin{Thm}	\label{c:minimal}
If $\Sigma^n  \subset \RR^{n+1}$ is a proper minimal hypersurface and $\Gamma_t$ is a mean curvature flow of hypersurfaces that are
 disjoint from $\Sigma$, then $ \Box_{\Gamma_t}  \, d^{ \frac{1}{n}}  \geq  0$ 
in the viscosity sense.
 \end{Thm}
 
  \vskip1mm
  Composition with an increasing concave function preserves supersolutions, so  
 $\Box_{\Gamma_t} \, d^p \geq 0$ for $p \in (0, 1/n]$ and  $\Box_{\Gamma_t} \log d \geq 0$.  This theorem is sharp when $n=1$ since
 $d$ is linear when $\Gamma_t = \Gamma$ is a static line segment and $\Sigma$ is a  line, so any higher power of $d$ would be strictly subharmonic.
 
In the case where $\Gamma_t$ is compact, taking the minimum of $d$ over $x \in \Gamma_t$ and applying Theorem \ref{c:minimal} yields the well-known result that the minimum distance from $\Gamma_t$ to $\Sigma$ is increasing in time. However, Theorem \ref{c:minimal} is considerably stronger: it provides a local version of this statement that holds even when $\Gamma_t$ is non-compact; it gives a local Harnack inequality:

\begin{Thm}	\label{t:apart}
Let $\Sigma$ and $\Gamma_t$ be as in Theorem \ref{c:minimal}.
Let $R > 0$ and $x_0 \in \RR^{n+1}$ be arbitrary and define $\phi = (R^2 - |x-x_0|^2 - 2\,n\,t)_+$.  Then we have for $t \geq t_0$ that
\begin{align}
	 d(x,t) \geq \phi^n (x,t) \, \,  \inf_{\Gamma_{t_0}} \,  \,  \frac{d } { \phi^n}  
	   \, .
\end{align}
\end{Thm}

 \vskip1mm
 When both $\Sigma$ and $\Gamma$ are minimal, we get the following corollary of Theorem \ref{c:minimal}:

 \begin{Cor}	\label{c:minimalC}
If $\Sigma^n  \subset \RR^{n+1}$ is a proper minimal hypersurface and $\Gamma$ is a minimal hypersurface that is
 disjoint from $\Sigma$, then $ \Delta_{\Gamma}  \, d^{ \frac{1}{n}}  \leq  0$ 
in the viscosity sense.
 \end{Cor}

 \vskip1mm
Combining Corollary \ref{c:minimalC} with the parabolicity of the two-dimensional plane implies  the    half-space theorem:

\begin{Thm}	\label{c:HM}
\cite{HM}
If $\Sigma \subset \{ x_3 > 0 \} \subset  \RR^3$ is a complete properly immersed minimal surface, then $\Sigma$ is a plane $\{ x_3 = c \}$.
\end{Thm}

\vskip1mm
The original proof of the   half-space theorem   involves sliding two-dimensional catenoids, which grow logarithmically,
 until they touch a potential counter-example. 
In higher dimensions,  hyperplanes are no longer parabolic and 
  the   half-space theorem does not hold.  This can be seen from the fact that higher dimensional catenoids are contained in slabs between two hyperplanes, \cite{Bl,CjH}.

   \vskip1mm
  When $\Gamma$ is a hyperplane, we get better estimates:

\begin{Thm}	\label{c:combine}
If $\Gamma$ is a hyperplane and $\Sigma$ is minimal,   then $| \nabla_{\Gamma} \, d|^2 \leq 1 -\epsilon_n$
and
$ \Delta_{\Gamma} \, d^{p} \leq 0$ in the viscosity sense for $p \leq \frac{ 1 + (n-1) \, \epsilon_n}{n}$, 
where $\epsilon_n =   [4\, (n-1)]^{-4} >0$.
\end{Thm}

\vskip1mm
Tracking the distance between two evolving hypersurfaces arises naturally in a variety of settings, even when one is ultimately concerned with a single flow $\Gamma_t$. A prominent example is the level set method for weak solutions of MCF (see, e.g. the expository book and articles \cite{G, CM2, CM3}). In this, one selects a function $u$ whose zero-level set coincides with $\Gamma_0$ and then considers the PDE under which every level set evolves by mean curvature  with initial data $u$. This leads to a nonlinear degenerate parabolic PDE; see  \cite{ES} and   \cite{ChGG}. By construction, distinct level sets of $u$ remain disjoint for all time.

\vskip1mm
We get a similar, but weaker, result when both
 $\Sigma_t$  and $\Gamma_t$ are flows:

 \begin{Thm}	\label{t:maind}
 If $\Gamma_t$ and $\Sigma_t$ are disjoint mean curvature flows where each $\Sigma_t$ is proper, then 
$\Box_{\Gamma_t} \,  d \geq - |\nabla_{\Gamma_t} d|^4$
in the viscosity sense and, thus, 
$\Box_{\Gamma_t}\,  \log d \geq 0$. 
\end{Thm}

\vskip1mm
When $\Gamma_t$ is compact,
 one can combine this theorem with the maximum principle to show that the minimum distance from $\Sigma_t$ to $\Gamma_t$ is increasing over time.  This is the well-known avoidance principle.
 Using Theorem \ref{t:maind}, we can show a local version of the avoidance principle that holds even when the flows are non-compact.  This shows that 
two disjoint flows push  apart, or at least do not come together very rapidly, even locally.

\begin{Thm}	\label{t:twoflowslocal}
Suppose that $\Sigma_t$ and $\Gamma_t$ are disjoint MCFs of hypersurfaces in $\RR^{n+1}$,  $\Sigma_t$ is proper, and $\Gamma_t$ is compact. 
 For $\kappa>0$ and any $x_0 \in \RR^{n+1}$,
define $\psi (x,t)= \e^{ - \kappa \, (|x- x_0|^2 + 2\,n\, t)}$.  Then we have for $t_1 \geq t_0$ and $x \in \Gamma_{t_1}$  that
\begin{align}	\label{e:TLF}
	d(x,t) \geq \psi (x,t) \, \int_{x \in \Gamma_{t_0}} \, \, \frac{ d(x,t_0)}{ \psi (x,t_0)} \, .
\end{align}
\end{Thm}

\vskip1mm
The estimate \eqr{e:TLF} is not as strong as Theorem \ref{t:apart} where $\Sigma_t$  is a   static minimal hypersurface.
This comes from the fact that we no longer get a power of $d$ to be a supersolution, but rather must go all the way to $\log d$.  This difference forces us to replace a power of 
$|x|^2 + 2n\, t$ by an exponential.

Theorem \ref{t:maind} should be compared with the estimate in  \cite{BK} for the separation function between two different locally defined
sheets of a single mean curvature flow.

 These estimates are nearly optimal, as we will see next. In particular, $d$ is not in general a supersolution to the heat equation, as the next proposition shows in the case where $\Sigma_t$ are shrinking spheres.  

 \begin{Pro}	\label{p:sphereX}
 For an evolving sphere $\Sigma_t^n \subset \RR^{n+1}$ of radius $r(t)$ centered at the origin
 \begin{align}	\label{e:D8}
 \Box \, d(x,t) &= 
 \begin{cases}
 \frac{1}{|x|}\,\left(\frac{n\,d}{d+|x|}-|\nabla^T\,d|^2\right) 
&\text{if $\Gamma_t$ is inside the sphere.  }\\
\frac{1}{|x|}\,\left(\frac{n\,d}{|x|-d}+|\nabla^T\,d|^2\right)
&\text{if $\Gamma_t$ is outside the sphere.  }
\end{cases}
\end{align}
When $\Gamma_t$ is inside $\Sigma_t$, $x_0 \in \Gamma_{t_0}$, and $x_0 \in T_{x_0} \Gamma_{t_0}$, then 
 $ \Box \, d (x_0 , t_0) < 0$ if $|x_0| > \frac{n-1}{n} \, r(t_0)$.
 \end{Pro}

 \vskip1mm
 Our approach begins by analyzing the hessian of the distance function at points where it is smooth, and then employs a barrier argument to study its behavior at non-smooth points in the spirit of \cite{Ca}. Alternatively, these questions can be addressed using variational methods and the two-point maximum principle (see \cite{A,ALM,BK,Br,DZ,HK,L}), but the method presented here is more  direct and yields stronger information.

    \section{The Lapacian of the distance function to a hypersurface}
    
 The triangle inequality implies that $d$ defined by \eqr{e:defd} is Lipschitz with $|d(x) - d(y)| \leq |x-y|$, and thus also continuous and differentiable almost everywhere by Rademacher's theorem ($3.1.6$ in \cite{F}).

 Define the set $\cD$ to be  the set of points  $p \in \RR^{n+1}$ so that 
 \begin{itemize}
 \item $d$ is differentiable at $p$, and 
 \item there is a unique closest point  $q = q_p$  in $\Sigma$.  
\end{itemize}
     Given $p \in \cD$, 
  let     $\nn (q) = \frac{q - p}{|q -p|}$ be the ``outward unit normal'' to $\Sigma$ at $q$.   Let   $\lambda_1 \leq \lambda_2 \leq \dots \leq \lambda_n$ the principal curvatures{\footnote{The principal curvatures are the eigenvalues of the second fundamental form $A$.
  The mean curvature $H$ is the trace of $A$, so $H = \sum_{i=1}^n \lambda_i$.}}
   of $\Sigma$ at $q$ and   set $\alpha_i = d(p) \, \lambda_i$.   Since $q$ is closest to $p$, we have $B_{d(p)} (p) \cap \Sigma = \emptyset$, and thus the second derivative test gives that 
\begin{align}
	- 1 \leq \alpha_1 \leq \alpha_2 \leq \cdots \leq \alpha_n   \, ,
\end{align}

\begin{Def}	\label{d:cD0}
We will say that $p \in \cD$ is in $\cD_0$  if $-1 < \alpha_1$. 
\end{Def}

 \begin{Lem}	\label{l:GTF}
 At any  $p \in \cD_0$,  $\nabla d (p)  = - \nn$, $d$ is twice differentiable and $d (p)  \, \Hess^{\RR^{n+1}}_d (p)$ has eigenvalues $\frac{\alpha_1}{1+\alpha_1} , \dots , \frac{\alpha_n}{1+\alpha_n}$ and $0$ in the direction $\nn$. In particular, 
  we have
\begin{align}  \label{e:old65}
	d\, \Delta^{\RR^{n+1}} \, d &=   \sum_{i=1}^n \frac{\alpha_i}{1+ \alpha_i}=   \sum_{i=1}^n\left( \alpha_i -  \frac{\alpha_i^2}{1+ \alpha_i} \right)= d\, H (q) - \sum_{i=1}^n \frac{\alpha_i^2}{1+ \alpha_i} \, .
\end{align}
\end{Lem}

\begin{proof}
At any point  $p \in \cD_0$, $d$ is twice differentiable, $\nabla d = - \nn$,  and the Euclidean hessian of $d$  has eigenvalues  (see page $355$ in \cite{GiTr}, cf. section $6$ in \cite{ES}) 
 \begin{align}	\label{e:lis}
 	\frac{\lambda_i}{1 + d \, \lambda_i} {\text{ for $i \leq n$ and }} 0 {\text{ in the direction $\nn$}}. 
 \end{align}
 Finally, 
  the last equality in \eqr{e:old65} uses that $H$ is the sum of the principal curvatures.
\end{proof}

\vskip1mm
We will use the following elementary fact below:  

\begin{Lem}	\label{l:s1s2}
If $p \in \cD_0$, then $\min_{i} \, \frac{\alpha_i}{1+\alpha_i} = \frac{\alpha_1}{1+\alpha_1}$. 
\end{Lem}

\begin{proof}
This follows from the fact that   the function $f(s) = \frac{s}{1+s}$ is increasing for $s> -1$ since
\begin{align}
	f'(s) = (1+ s)^{-2} > 0 \, .
\end{align}
\end{proof}

\vskip1mm
Suppose now that $\Gamma \subset \RR^{n+1}$ is a   hypersurface that is disjoint from $\Sigma$.

\begin{Pro}	\label{p:mainset}
If $p \in \cD_0 \cap \Gamma$, then we have at $p$ that 
\begin{align}  
	d\, \Delta_{\Gamma} \, d  - d \, H 
	+ d\, H_{\Gamma} \, \langle \nn  , \nn_{\Gamma} \rangle
	+ \sum_{i=1}^n \frac{\alpha_i^2}{1+ \alpha_i}&=     - d \, \Hess^{\RR^{n+1}}_d (\nn_{\Gamma}^T , \nn_{\Gamma}^T)  \leq    - \frac{\alpha_1}{1+\alpha_1} \, |\nabla_{\Gamma} d|^2 
	 \, ,  \notag
\end{align}
where $\nn_{\Gamma}^T$ is the projection of $\nn_{\Gamma}$ perpendicular to $\nn$.
 \end{Pro}
 
 \begin{proof}
Let $\bar{e}_i$ be an orthonormal frame at $p$, so that at $p$ 
\begin{align} 
	d\, \Delta_{\Gamma} \, d  &=  d\, H_{\Gamma} \, \langle \nabla^{\RR^{n+1}} d , \nn_{\Gamma} \rangle +
	 \sum_{i=1}^n d\, \Hess^{ \RR^{n+1}} (\bar{e}_i , \bar{e}_i) \notag \\
	 &= -  d\, H_{\Gamma} \, \langle \nn  , \nn_{\Gamma} \rangle + 
	d\, \Delta^{\RR^{n+1}} d - d \, \Hess^{\RR^{n+1}}_d (\nn_{\Gamma} ,\nn_{\Gamma}) \\
	&=  -  d\, H_{\Gamma} \, \langle \nn  , \nn_{\Gamma} \rangle + d \, H  - \sum_{i=1}^n \frac{\alpha_i^2}{1+ \alpha_i}  - d \, \Hess^{\RR^{n+1}}_d (\nn_{\Gamma} , \nn_{\Gamma})  \notag \, , 
\end{align}
where the second equality used that the Euclidean gradient of  $d$ is $-\nn$ and the last equality used Lemma \ref{l:GTF}.
The first equality in the proposition  follows from this and the fact that the hessian of $d$ vanishes in the $\nn$ direction.

Using again that the Euclidean gradient of  $d$ is $-\nn$, we see that
\begin{align}
	|\nn_{\Gamma}^T|^2 = 1 - \langle \nn_{\Gamma} , \nn \rangle^2 = 1-  \langle \nn_{\Gamma} , \nabla^{\RR^{n+1}} \, d  \rangle^2 = |\nabla_{\Gamma} d|^2 \, ,
\end{align}
where $\nabla_{\Gamma} d$ is the tangential gradient of $d$ on $\Gamma$.
Therefore, since $\nn_{\Gamma}^T$ is contained in the span of the $e_i$'s, the variational characterization of eigenvalues 
and Lemma \ref{l:s1s2} give that
\begin{align}
		-d\, \Hess^{\RR^{n+1}}_d  (\nn_{\Gamma}^T , \nn_{\Gamma}^T)  \leq -\frac{\alpha_1}{1+ \alpha_1} \, \left| \nn_{\Gamma}^T \right|^2 = 
		- \frac{\alpha_1}{1+ \alpha_1} \, \left| \nabla_{\Gamma} d \right|^2 \, .
\end{align}
This gives the last inequality.
\end{proof}

\subsection{Minimality}

Suppose now that $\Sigma$ is minimal.     Theorem \ref{t:ndcor}
 gives two  upper bounds for the Laplacian of $d$ on a second hypersurface $\Gamma$.

   \begin{Thm}	\label{t:ndcor}
There exists $\kappa > 0$ depending only on $n$ so that on $\Gamma \cap \cD_0$
\begin{align}
	d \, \Delta_{\Gamma} \, d + d \, H_{\Gamma} \, \langle \nn , \nn_{\Gamma} \rangle  \leq  \left( \frac{2n-1}{2n} \right) |\nabla_{\Gamma} d|^4  - \kappa \, |A|^2 \, d^2 \, ,
\end{align}
where $|A|^2$ is evaluated at the unique closest point in $\Sigma$.  We also have
\begin{align}	\label{e:alsohave}
	d \, \Delta_{\Gamma} \, d + d \, H_{\Gamma} \, \langle \nn , \nn_{\Gamma} \rangle 
	 \leq  \left( \frac{n-1}{n} \right) |\nabla_{\Gamma} d|^4    \, .
\end{align}
\end{Thm}

 We will use the next two lemmas in the proof.

\begin{Lem}	\label{l:1802}
If $\sum_{i=2}^n x_i = L > 0$ and every $x_i > -1$, then
\begin{align}
	\sum_{i=2}^n \frac{x_i^2}{1+x_i} \geq \frac{L^2}{n-1 + L} \, ,
\end{align}
with equality holding if and only if every $x_i = \frac{L}{n-1}$.
\end{Lem}

\begin{proof}
Observe that 
\begin{align}
	\frac{s^2}{1+s} = \frac{s^2 +s}{1+s} - \frac{s}{1+s} =s - 1 + \frac{1}{1+s} \, .
\end{align}
Therefore,  we are looking for the minimum of 
\begin{align}
	f(x_2 , \dots , x_n) = \sum_{i=2}^n \frac{x_i^2}{1+x_i} = \sum_{i=2}^n \left( x_i - 1 + \frac{1}{1+x_i} \right)
	= L - (n-1) +  \sum_{i=2}^n   \frac{1}{1+x_i} \, ,
\end{align}
subject to the constraints that $\sum_{i=2}^n x_i = L > 0$ and every $x_i > - 1$.  
The next observation is that $f \to \infty$ as any $x_i \to -1$. Combining this with the constraint shows that $f \to \infty$ as we approach the boundary of the region.  It follows that $f$ must have an interior minimum (subject to the constraint). Namely, we have a Lagrange multiplier problem
\begin{align}
	{\text{minimize }} f {\text{ subject to the constraint }} g(x_2 , \dots , x_n) = \sum_{i=2}^n x_i = L \, .
\end{align}
Since the gradient of the constraint function $g$  is the vector of all ones, it follows that the partials of $f$ must all be equal at a constrained minimum.  However, $f_i = - (1+x_i)^{-2}$, so this means that the $x_i$'s are all equal to $L/(n-1)$ at the minimum.  This gives 
\begin{align}
	f \geq  \sum_{i=2}^n \frac{ (L/(n-1))^2}{1 + L/(n-1)} = \frac{ L^2}{n-1+L} \, , 
\end{align}
completing the proof.
\end{proof}

Motivated by this, given $b\leq 1$ and $\beta \geq 2$, define $F=F_{b,\beta}$ of $s \in [0,1)$ by
\begin{align}	\label{e:defFb}
	F(s) = b^2 \, \frac{s}{1-s} - \frac{s^2}{1-s} - \frac{s^2}{\beta-1 +s} \,  . 
\end{align}
Set $c  = \sqrt{1 - b^2}$.

\begin{Lem}	\label{l:Fn}
The supremum of $F$ defined by \eqr{e:defFb} over $s \in [0,1)$ is given by
\begin{align}
	\sup_{s \in [0,1)} F(s) =  F \left( (\beta-1) \, \frac{ 1-c}{\beta-1 +c} \right) =  \left( \frac{\beta-1}{\beta} \right) \, (1-c)^2 \leq 
	 \left( \frac{\beta-1}{\beta} \right) \,  b^4 \, .
\end{align}
Moreover, if $b<1$, then the supremum is achieved in $[0,1)$.
\end{Lem}

\begin{proof}
Using that $b^2 = 1-c^2$, we can rewrite  $F$ as
\begin{align}
	F(s) &= 
	   c^2  \,\left(  1 - \frac{1}{1-s} \right) + (\beta-1) -  \frac{  (\beta-1)^2  }{\beta-1 +s} 
	 \, .  
\end{align}
Differentiating this gives
\begin{align}
	F' (s) =  \frac{-c^2}{(1-s)^2}   + \frac{(\beta-1)^2}{(\beta-1+s)^2} \, .
\end{align}
From this, we see that $F$ has just one critical point $s_0= (\beta-1) \, \frac{ 1-c}{\beta-1 +c}$ in $[0,1]$ and 
\begin{align}
	0 < F' (s)   {\text{ for }} s < s_0 ,  {\text{ while }} F'(s) < 0 {\text{ for }} s_0 < s \, .
\end{align}
It follows that the maximum of $F$ is at $s_0$.  
Thus,  we have that 
\begin{align}
	\sup_{s\in [0,1)} \, F(s) &= F(s_0) =  \left( \frac{\beta-1}{\beta} \right) \, (1-c)^2 \, ,
\end{align}
where the second equality is easy to check using that
\begin{align}
	1-s_0=\frac{\beta\, c}{\beta-1+c}\,  {\text{ and }} 
	\beta-1+s_0=\frac{\beta\,(\beta-1)}{\beta-1+c}\,  .
\end{align}
Finally,     $s_0 = 1$ if and only if $b=1$, so 
 the supremum is a maximum
 if $b < 1$. 
 \end{proof}

\begin{proof}[Proof of Theorem \ref{t:ndcor}]
Applying Proposition \ref{p:mainset} and using that $H_{\Sigma}= 0$ gives that 
\begin{align}	\label{e:deltadd}
	d\, \Delta_{\Gamma} \, d + d \, H_{\Gamma} \, \langle \nn , \nn_{\Gamma} \rangle 
	&\leq - |\nabla_{\Gamma} d|^2 \, \frac{\alpha_1}{1+\alpha_1} - \sum_{i=1}^n \frac{\alpha_i^2}{1+ \alpha_i} \notag \\
	&=
	-\left(  |\nabla_{\Gamma} d|^2 \, \frac{\alpha_1}{1+\alpha_1} +  \frac{\alpha_1^2}{1+ \alpha_1} 
	\right) - \sum_{i=2}^n \frac{\alpha_i^2}{1+ \alpha_i} \, .
\end{align}
By minimality of $\Sigma$, $\alpha_1 \leq 0$ and $\sum_{i=2}^n \alpha_i = - \alpha_1$.  Thus, 
  Lemma \ref{l:1802} gives that
\begin{align}	\label{e:from1802}
	- \sum_{i=2}^n \frac{\alpha_i^2}{1+ \alpha_i}  \leq - \frac{ \alpha_1^2}{n-1-\alpha_1} \, . 
\end{align}
Setting $s = |\alpha_1|$ and using \eqr{e:from1802} in 
 \eqr{e:deltadd}, we get that 
\begin{align}	\label{e:usee}
	d\, \Delta_{\Gamma} \, d 
	+ d \, H_{\Gamma} \, \langle \nn , \nn_{\Gamma} \rangle  \leq
	   |\nabla_{\Gamma} d|^2 \, \frac{s}{1-s} -  \frac{s^2}{1-s} 
	  -  \frac{s^2}{n-1+ s}
	  \, .
\end{align}
The second claim \eqr{e:alsohave} follows from this and   Lemma \ref{l:Fn} with $b = |\nabla_{\Gamma} d|$ and $\beta = n$.

To prove the first claim, use  Lemma \ref{l:Fn} with $b = |\nabla_{\Gamma} d|$ and $\beta = 2n$ to get
\begin{align}
	|\nabla_{\Gamma} d|^2 \, \frac{s}{1-s} - \frac{s^2}{1-s} - \frac{s^2}{2n-1 +s}  = F_{b,\beta}(s) \leq \frac{2n-1}{2n} \,  |\nabla_{\Gamma} d|^4  \, .
\end{align}
Using this in the bound \eqr{e:usee} gives that 
 \begin{align}	 
	 	d\, \Delta_{\Gamma} \, d 
		+ d \, H_{\Gamma} \, \langle \nn , \nn_{\Gamma} \rangle  &\leq  \frac{2n-1}{2n} \,  |\nabla_{\Gamma} d|^4   + \left(  \frac{s^2}{2n-1 +s} - \frac{s^2}{n-1 +s} \right) \notag \\
		&\leq \frac{2n-1}{2n} \,  |\nabla_{\Gamma} d|^4   -   \frac{n \, s^2}{(2n-1 +s)\, (n-1 +s)} 
		\leq  \frac{2n-1}{2n} \,  |\nabla_{\Gamma} d|^4   -   \frac{s^2}{2n} 
				 \, . \notag
\end{align}
 Finally, since $\lambda_1$ is the most negative eigenvalue and the eigenvalues sum to zero (by minimality of $\Sigma$), there is a dimensional constant $C$ so that $|A|^2 \leq C \, \lambda_1^2$.   
\end{proof}

\section{Viscosity solutions}

We turn  to the proof of Theorem \ref{c:minimal}.  We will actually prove a stronger statement given in Theorem \ref{t:visc} below.
To state this 
 given a point $x$, define $\cA (x)$ by
\begin{align}
	\cA (x) =  \sup \, \{  |A| (y)\, | \, y \in \Sigma {\text{ with }} |x-y| = d(x) \} \, .
\end{align}
Using the convention on the sign of $H$ here, the mean curvature flow (MCF) $\Gamma_t$ satisfies
\begin{align}
	\partial_t \, x = H_{\Gamma_t} \, \nn_{\Gamma_t} \, .	\label{e:MCFGamma}
\end{align}

\begin{Thm}	\label{t:visc}
Suppose that $\Sigma$ is minimal and $\Gamma$ is a MCF.
There exists $\kappa > 0$ depending only on $n$, so that 
\begin{align}	\label{e:visc1}
	  -d\, \Box_{\Gamma} \,  d    \leq  \frac{2n-1}{2n} \, |\nabla_{\Gamma}  d|^4   - \kappa \,  \cA^2    \, d^2
\end{align}
 in the viscosity sense.  Moreover, we also have in the viscosity sense that
 \begin{align}	\label{e:visc1AA}
	  -d\, \Box_{\Gamma} \,  d    \leq  \frac{n-1}{n} \, |\nabla_{\Gamma}  d|^4   \, .
\end{align}

\end{Thm}

 We will use a barrier argument in the spirit of \cite{Ca}, cf. \cite{CrIL}.

\begin{Lem}	\label{l:translate}
Suppose that $p \notin \Sigma$ and $q \in \Sigma$ satisfies $d(p) = |p-q|$. Given $\epsilon \in (0,d(p))$, let $\Sigma_{\epsilon}$ be the translated hypersurface
\begin{align}
	\Sigma_{\epsilon} = \Sigma + \epsilon \, \frac{p-q}{|p-q|}  
\end{align}
and $d_{\epsilon}$   the distance function to $\Sigma_{\epsilon}$. If  $\cD^{\epsilon}_0$ is the $\cD_0$ for $d_{\epsilon}$, then $p \in \cD_0^{\epsilon}$.
\end{Lem}

\begin{proof}
Since $q$ is the closest point in $\Sigma$ to $p$, the first   derivative test  gives that $\partial B_{|p-q|}(p)$ is tangent to $\Sigma$ at $q$ and, moreover, 
$B_{|p-q|}(p) \cap \Sigma = \emptyset$.   This implies that $q$ is a smooth point of $\Sigma$ (i.e., it is embedded there).  
The second derivative test implies that 
\begin{align}	\label{e:x1d}
	0 \geq \lambda_1 \, d(p) \geq - 1 \, ,
\end{align}
where $\lambda_1$ is the lowest principal curvature for $\Sigma$ at $q$.

Since proper subsegments of minimizers are uniquely minimizing, $p$ is inside the normal exponential tube for $\Sigma_{\epsilon}$ for any $\epsilon > 0$ and, thus, $d_{\epsilon}$ is smooth at $p$  (cf. \cite{HKa}).  Moreover, the translate of $q$ is the unique closest point.  It remains to check that $p \in \cD^{\epsilon}_0$. 
However, this follows from \eqr{e:x1d} since $d^{\epsilon} (p) < d(p)$. 
\end{proof}

\vskip1mm
We turn next to the proof of Theorem \ref{t:visc}.  Before doing so, it is useful to recall the notion of a viscosity supersolution. To show that $u$ satisfies
\begin{align}
	- \Box \, u \leq  F(u,\nabla u) \notag
\end{align}
in the viscosity sense, we must show the following:  For any point $(x_0 , t_0)$ and any $C^2$ function $\phi (x,t)$ that satisfies
$$
	\phi (x,t) \leq u(x,t) {\text{ for }} t \leq t_0 {\text{ with equality at }} (x_0 , t_0) \, , 
$$
then we must have 
$$
	- \Box \, \phi (x_0 , t_0)  \leq  F(u(x_0 ,t_0),\nabla u (x_0 , t_0))  \, .
$$
The function $\phi$   is called a parabolic viscosity barrier.

\vskip1mm

\begin{proof}[Proof of Theorem \ref{t:visc}]
The two claims follow similarly using  Theorem \ref{t:ndcor}; we will give the details for the first claim.

Let  $\kappa = \kappa (n) > 0$ be given by Theorem \ref{t:ndcor}.    Fix  a time $t_0$, set $\Gamma = \Gamma_{t_0}$, and fix $x_0 \in \Gamma$. Let $y_0$ be a closest point in $\Sigma$ to $x_0$ (this exists  since $d$ is continuous 
and  $\Sigma$ is proper).  Given $\epsilon \in (0, d(x_0))$,  define $\Sigma_{\epsilon}$, 
   $d_{\epsilon}$, and $\cD^{\epsilon}_0$   as in Lemma \ref{l:translate}, so  that $x_0 \in \cD^{\epsilon}_0$, 
\begin{align}	\label{e:de1}
	d_{\epsilon} (x_0) = d (x_0) - \epsilon \, .
\end{align}
Since every point in $\Sigma_{\epsilon}$ is distance at most $\epsilon$ to $\Sigma$, we also have
\begin{align}	\label{e:de2}
	d(x) = \dist (x, \Sigma) \leq  d_{\epsilon} (x) + \epsilon  \, .
\end{align}
  Thus, Theorem \ref{t:ndcor} gives at $x_0$ that 
\begin{align}	\label{e:visc1e}
	d_{\epsilon} \, \Delta \, d_{\epsilon}   + d_{\epsilon} \, H_{\Gamma} \,  \langle \nn , \nn_{\Gamma} \rangle \leq \left( \frac{2n-1}{2n} \right) \, |\nabla d_{\epsilon}|^4    - \kappa \,   A^2(y_0)  \, d_{\epsilon}^2  \, .
\end{align}
Set $h_{\epsilon} = d_{\epsilon} + \epsilon$ (cf. \cite{Ca} for the distance to a point).

As in \cite{CrIL}, let $\phi$ be a parabolic viscosity barrier for $d$ at $x_0$, so that $\phi$ is $C^2$, $\phi (x_0,t_0) = d(x_0,t_0)$ and $\phi (x,t)\leq d (x,t)$  for $t\leq t_0$ and $x$ in a neighborhood $U$ of $x_0$. 
Using \eqr{e:de1} and \eqr{e:de2}, we see that
\begin{align}
	\phi (x_0,t_0) = h_{\epsilon} (x_0) {\text{ and }} \phi (x,t_0)\leq h_{\epsilon} (x) {\text{ in }} U \, .
\end{align}
It follows that the $C^2$ function $h_{\epsilon}(x) - \phi (x,t_0)$ has a local minimum at $x_0$, so the first and second derivative tests give that
$\nabla_{\Gamma} \phi   (x_0,t_0)= \nabla_{\Gamma} h_{\epsilon} (x_0) = \nabla_{\Gamma} d_{\epsilon} (x_0. t_0)$ and 
\begin{align}
	\Delta_{\Gamma}  \phi  (x_0,t_0) &\leq \Delta_{\Gamma} h_{\epsilon} (x_0)   =  \Delta_{\Gamma} d_{\epsilon}  (x_0,t_0) \, .
\end{align}
Combining this with \eqr{e:visc1e} gives at $x_0$  that
\begin{align}
	 \Delta_{\Gamma} \,  \phi    &\leq  -  d_{\epsilon} \, H_{\Gamma} \,  \langle \nn , \nn_{\Gamma} \rangle +
	  \left( \frac{2n-1}{2n} \right) \, \frac{ |\nabla_{\Gamma} d_{\epsilon}|^4}{d_{\epsilon}}  
	      - \kappa \,   A^2(y_0)  \, d_{\epsilon} \notag \\
	     & =  -(\phi - \epsilon) \, H_{\Gamma} \,  \langle \nn , \nn_{\Gamma} \rangle  + \left( \frac{2n-1}{2n} \right) \, \frac{ |\nabla_{\Gamma}  \phi|^4}{\phi - \epsilon}       - \kappa \,   A^2(y_0)  \, (\phi -\epsilon)
	\, .
\end{align}
Since this holds for every $\epsilon > 0$, we can take the limit as $\epsilon \to 0$ to get 
\begin{align}	\label{e:uphered}
	 \Delta_{\Gamma} \,   \phi    &\leq    -\phi  \, H_{\Gamma} \,  \langle \nn , \nn_{\Gamma} \rangle  +  \left( \frac{2n-1}{2n} \right) \, \frac{ |\nabla_{\Gamma}  \phi|^4}{\phi}       - \kappa \,   A^2(y_0)  \,\phi 
	\, .
\end{align}
Let $x(t)$ be the flow line through $x_0$ with $x(t_0) = x_0$, so  that $d(x(t) , t) \leq |x(t) - y_0|$ with equality at $t=t_0$.  Thus, we see for $t \leq t_0$  that
\begin{align}
	\phi (x(t) , t)\leq d(x(t) , t) \leq |x(t) - y_0| {\text{ with equality at }} t = t_0 \, .
\end{align}
From this, we conclude that 
 \begin{align}
 	\phi_t (x_0 ,t_0) &\geq \partial_t \big|_{t=t_0} \, |x(t) - y_0| = \langle x_t (t_0) - y_0 , \frac{ x(t_0) - y_0}{d} \rangle
	 =       - H_{\Gamma_{t_0}} \,  \langle \nn_{\Gamma_{t_0}} , \nn \rangle \, ,
 \end{align}
 where the last equality on the second line used the mean curvature flow equation.  Combining this with \eqr{e:uphered} gives
\begin{align}	 
	- \Box_{\Gamma} \,   \phi    &\leq    \left( \frac{2n-1}{2n} \right) \, \frac{ |\nabla_{\Gamma}  \phi|^4}{\phi}       - \kappa \,   A^2(y_0)  \,\phi 
	\, .
\end{align}
Since this holds for any closest point $y_0$, this completes the proof.
\end{proof}

\begin{proof}[Proof of Theorem \ref{c:minimal}] Given $p \in (0 , \frac{1}{n}]$, 
Theorem \ref{t:visc} and the chain rule give that
\begin{align}
	-\frac{1}{p} \, \Box_{\Gamma} \,  d^p = d^{p-2} \, \left( d \, \Box_{\Gamma} \, d + (p-1) \, |\nabla_{\Gamma} d|^2
	\right)
	  \leq d^{p-2} \, |\nabla_{\Gamma} d|^2 \, \left( p - \frac{1}{n} \right) \leq 0 \, .
\end{align}
\end{proof}

We specialize the next corollary to the case where $\Gamma$ is minimal:

\begin{Cor}	\label{c:cutoff}
There exists $\kappa > 0$ depending only on $n$ so that  if $\Gamma$ is minimal and 
$\phi$ has  compact  support, then
\begin{align}
	\int  \phi^2 \, \left(   |\nabla_{\Gamma} \log d|^2  + 4n \, \kappa \, \cA^2 \right)  &\leq 16 \,  n^2 \, \int |\nabla_{\Gamma} \phi|^2 \, .
\end{align}
\end{Cor}

\begin{proof}
Theorem \ref{t:visc} and the chain rule give that 
\begin{align}	\label{e:visc1AAa}
	  \Delta_{\Gamma} \, \log d    \leq - \frac{1}{2n} \, |\nabla_{\Gamma} \log d|^2   - \kappa \,  \cA^2    
\end{align}
in the viscosity sense and thus also in the sense of distributions, \cite{I}. 
The divergence theorem and  \eqr{e:visc1AAa} give that
\begin{align}
	\int  \phi^2 \, \left(  \frac{|\nabla_{\Gamma} \log d|^2}{2n} + \kappa \, \cA^2 \right)  &\leq -\int \phi^2 \, \Delta_{\Gamma}\,  \log d \leq  2 \, \int |\phi | \, |\nabla_{\Gamma} \phi| \, |\nabla_{\Gamma} \log d| \notag
	\\
	&\leq \frac{1}{4n} \, \int \phi^2 \, |\nabla_{\Gamma} \log d|^2 + 4n \, \int |\nabla_{\Gamma} \phi|^2 \, .
\end{align}
\end{proof}

 \section{Minimal surfaces in a half-space}

We turn now to the case of minimal surfaces in $\RR^3$ and the  half-space theorem of Hoffman and Meeks, \cite{HM,MP,P}. 
The original proof used catenoid barriers and the maximum principle. We see next that it is also a consequence of the differential inequality for the distance function. 

\begin{proof}[Proof of Theorem \ref{c:HM}]
By assumption,  $\Sigma \subset \{ x_3 > 0 \}$ is a proper minimal immersion.  It follows that $d > 0$ on $\Gamma =\{ x_3 = 0 \}$.   On the other hand, Corollary  \ref{c:minimalC}
gives that 
	 $\Delta_{\Gamma} \,   \sqrt{d} \leq 0$ in the viscosity sense and, thus, $  \sqrt{d}$ is a positive superharmonic function on the plane.  Since the plane is parabolic (see Lemma \ref{l:paraR2} below),   $ d \equiv r$ must be  constant on $\Gamma$.     It follows that:
\begin{enumerate}
\item  For each $p \in \Gamma$, 
there is at least one point in   $\overline{B_{r} (p)} \cap \Sigma \cap \{ x_3 > 0 \}$.   
\item $\Sigma$ does not intersect $\{ x_3 < r \}$.
\end{enumerate}
The second claim (2) follows since this set is too close to $\Gamma$.  Combining (1) and (2) gives that 
 $\Sigma$ must  be the parallel plane $\{ x_3 = r \}$. 
\end{proof}

We used the parabolicity of $\RR^2$ in the proof; for completeness, we state the form of this that we need below:

\begin{Lem}	\label{l:paraR2}
If $u > 0$ is a continuous function on $\RR^2$ that satisfies $\Delta \, u \leq 0$ in the viscosity sense, then $u$ is constant.
\end{Lem}

\begin{proof}
The classical case where $u$ is 
  twice differentiable follows from the logarithmic cutoff trick (see, e.g., Proposition $1.37$ in \cite{CM1}).

Suppose that $u>0$ is continuous and  $\Delta \, u \leq 0$ in the viscosity sense and, thus, also in the distributional sense (by theorem $1$ in \cite{I}).  Let $\psi \geq 0$ be a smooth function with compact support and integral one
and then let $u_{\psi} = u \ast \psi$ be the convolution of $u$ and $\psi$.  It follows that $u_{\psi}$ is smooth and positive.  
Since $\psi \geq 0$ and $\Delta \, u \leq 0$ in the distributional sense, it follows that $\Delta u_{\psi} = (\Delta \, u) \ast \psi \leq 0$ as well. By the first part of the proof, we see that $u_{\psi}$ must be constant.  We now take a sequence of $\psi$'s that converges to the $\delta$ function
(i.e., an approximate identity) to see that $u$ itself is constant.
\end{proof}

 \section{Improved gradient estimate}
 
 In this section, $\Sigma^n \subset \{ x_{n+1} > 0 \} \subset \RR^{n+1}$ is a proper immersed minimal hypersurface in a half-space.   
  We will be interested in the restriction of $d$ to $\Gamma = \{ x_{n+1} = 0 \}$. 
 
  \vskip1mm
 Given any distance function, the norm of  its gradient is at most one.  
 The next result  is an improved gradient estimate, showing that  $|\nabla_{\Gamma} d|$ is strictly below one on $\Gamma \cap \cD_0$.
  
 \begin{Thm}	\label{p:minimal}
 At each  point in $d \in \Gamma \cap \cD_0$, we have that  $|\nabla_{\Gamma} d|^2 \leq (1-\epsilon_n) < 1$
  with $\epsilon_n =   \frac{1}{[4\, (n-1)]^4} >0$.
 \end{Thm}
 
In this special case where $\Gamma = \{ x_{n+1} = 0 \}$, Lemma \ref{l:GTF} give for $p \in \Gamma \cap \cD_0$ that 
 \begin{align}	\label{e:nabdF}
 	|\nabla_{\Gamma}  d|^2 (p) = 1 - \langle \partial_{n+1} , \frac{q-p}{|q-p|}  \rangle^2 \, ,
 \end{align}
where $q \in \Sigma$ is the unique closest point to $p$.  Thus, 
  Theorem \ref{p:minimal} follows from showing that 
\begin{align}	\label{e:eus2}
	\langle \partial_{n+1} ,  q-p \rangle^2 \geq  \epsilon_n  \, | q-p|^2 \, .
\end{align}

In the next lemma, we let $z = (x_1 , \dots, x_{n-1})$ be the first $n-1$ coordinates, so $x= (z, x_n , x_{n+1} )$.  Given $s>0$, we will 
set   $p' = (  0 , \dots , 0,s,s)$ define a function $f$ by
\begin{align}	\label{e:definef}
	f (x) = |x-p'|^2 - \frac{n}{n-1} \, |z|^2 \, .
\end{align}
$\Sigma$ is contained in the half-space $\{ x_{n+1} > 0 \}$.  The next lemma shows that $f$ has nice properties when $\Sigma$ is locally
contained in a quarter-space $\{ x_n , x_{n+1} > 0 \}$.

\begin{Lem}	\label{l:function2}
If  $B_s \cap \Sigma^n \subset   \{ x_{n} > 0 \}$, then 
  $f$ from \eqr{e:definef}  satisfies  $\Delta_{\Sigma} f \geq 0$   and 
\begin{align}
	\max_{\partial (B_s \cap \Sigma)} f \leq \left( 2 - \frac{1}{n-1}\right) \, s^2  \, .
\end{align}
\end{Lem}

\begin{proof}
 The coordinate functions are harmonic and $\Delta_{\Sigma} |x|^2 = 2\, n$, so we have that
\begin{align}
	\Delta_{\Sigma} \, f = 2\, n -  \frac{n}{n-1} \, \Delta_{\Sigma} \, |z|^2 = 2\, n - \frac{n}{n-1} \, \sum_{i=1}^{n-1} 2\, |\nabla_{\Sigma} x_i|^2 \geq 2\, n -
	\frac{2\, n}{n-1} \, (n-1) 
	 = 0 \, .
\end{align}
This gives the first claim.  Suppose now that $(z,x_n , x_{n+1}) \in \partial (B_s \cap \Sigma)$.  It follows that $x_n , x_{n+1} \geq 0$ and
$x_n^2 + x_{n+1}^2 +|z|^2 =s^2$.   We have that
\begin{align}
	f (z,x_n , x_{n+1}) &=  (s-x_n)^2 + (s - x_{n+1})^2 - \frac{1}{n-1} \, |z|^2 \notag \\
	&= 2\, s^2 - 2\, s \, x_n + x_n^2    - 2 \, s\, x_{n+1} + x_{n+1}^2
	-  \frac{1}{n-1} \,  \left( s^2- x_n^2 - x_{n+1}^2
	\right) \\
	&= \left( 2 - \frac{1}{n-1} \right)\, s^2 + x_n \, \left( \frac{n\, x_n }{n-1}   - 2\, s\right) + x_{n+1} \, \left( \frac{n\,  x_{n+1}}{n-1} - 2\, s \right) 
	\, . \notag
\end{align}
Finally, the last two terms are nonpositive since $x_n , x_{n+1}$ are both in $[0,s]$.
\end{proof}

  The next lemma shows that if we can find a closest point in $\Sigma$ where the angle is low enough, then we can use the previous lemma to get a contradiction.

\begin{Lem}	\label{l:findquarter}
Suppose that there exists $p \in \{ x_{n+1} = 0 \}$ and $q \in \Sigma$ with $d(p) = |p-q|$ and 
\begin{align}
	\langle \partial_{n+1} , q-p \rangle  < \frac{|q-p|}{16\, (n-1)^2} \, .
\end{align}
Then we can translate and rotate in the first $n$ coordinates so that 
\begin{enumerate}
\item  $B_s \cap \Sigma \subset \{ x_n , x_{n+1} > 0\}$ where $s = \frac{d(p)}{4 \, (n-1)}$.
\item The function $f$ from \eqr{e:definef} has  
 $f(q) > s^2 \left(2- \frac{1}{n-1} \right)$.
 \item $q$ is in the interior of $B_s \cap \Sigma$.
 \end{enumerate}
\end{Lem}

\begin{proof}
Set $r = |q-p|$, so that $s =  \frac{r}{4 \, (n-1)}$. Let $q'$ be the projection of $q$ to $\{ x_{n+1} = 0 \}$.
We translate and rotate in the first $n$ coordinates so that the origin lies along the line segment from $p$ to $q'$, this segment is in the $x_n$-axis, and  
\begin{align}
	  x_{n} (q') = \frac{r}{16 \, (n-1)^2} \, .
\end{align}
In these coordinates $z(q) = 0$, $x_n (q) = \frac{r}{16 \, (n-1)^2}$ and $0 < x_{n+1} (q) < \frac{r}{16\, (n-1)^2}$. It follows immediately that $q$ is in the interior of
$B_s$, giving (3).  Moreover, we see that 
\begin{align}
	f(q) &= \left( s - x_n (q) \right)^2 + \left( s - x_n (q) \right)^2 - \frac{1}{n-1} \, |z(q)|^2 \geq 
	2\,  \left( s -  \frac{r}{16 \, (n-1)^2} \right)^2  \notag \\
	&>  2\, s^2 - 4 \,s \, \frac{r}{16 \, (n-1)^2} = 2\, s^2 - 4 \,s \, \frac{4\, (n-1) \, s}{16 \, (n-1)^2} = \left( 2 - \frac{1}{n-1} \right) \, s^2 \, .
\end{align}
This gives (2).  
It remains to prove (1), i.e., to show that $B_s \cap \Sigma \subset \{ x_n > 0 \}$. 
Suppose  that $x \in B_s \cap \Sigma$. Set $\lambda = \frac{1}{16\, (n-1)^2}$, so that $s= \sqrt{\lambda} \, r$ and $p$ has coordinates
\begin{align}	\label{e:usemepn}
	z(p) = 0 \, ,  (\lambda - 1) \, r \leq x_{n} (p) \leq  (2\, \lambda - 1) \, r {\text{ and }} x_{n+1} (p) = 0 \, .
\end{align}
Since $B_r (p) \cap \Sigma = \emptyset$, we have for any $x \in \Sigma$ that
\begin{align}
	r^2 &\leq |x - p|^2 = |z|^2 + (x_n -p_n)^2 + x_{n+1}^2 = |x|^2 -  2\, x_n \, p_n + p_n^2 \leq |x|^2 -  2\, x_n \, p_n + (\lambda - 1)^2\, r^2 
	 \, , \notag
\end{align}
where the last inequality used \eqr{e:usemepn}.
If $x \in B_s \cap \Sigma$, this implies that
\begin{align}
	0 &\leq s^2  -  2\, x_n \, p_n + (\lambda - 1)^2\, r^2 - r^2  =  \lambda \, r^2 -  2\, x_n \, p_n + (\lambda^2 - 2\, \lambda)\, r^2   \leq 
	-  2\, x_n \,p_n 		 \, ,
\end{align}
where the inequality used that $0< \lambda < 1$. Since $p_n < 0$, this implies that $x_n > 0$, completing the proof.
\end{proof}

\begin{proof}[Proof of Theorem \ref{p:minimal}]
Suppose that $p \in \Gamma$ and $q \in \Sigma$ satisfy
$d(p) = |q-p|$. If 
\begin{align}	 
	\langle \partial_{n+1} , q-p \rangle  < \frac{|q-p|}{16\, (n-1)^2} \, ,
\end{align}
then we get a contradiction from 
  Lemmas \ref{l:function2}  and \ref{l:findquarter} and
  the maximum principle.  Thus, we see that
   \eqr{e:eus2} holds with $\epsilon_n =   \frac{1}{[4\, (n-1)]^4}$,   giving the claim.
\end{proof}

\begin{proof}[Proof of Theorem \ref{c:combine}]
If $p \in \cD_0 \cap \Gamma$, then Theorem \ref{p:minimal} gives at $p$  that
\begin{align}		\label{e:whatn}
	|\nabla_{\Gamma} d|^2 \leq 1 - \epsilon_n \, .
\end{align}
In $p \notin \cD_0$, then we proceed as in the proof of 
Theorem \ref{t:visc}: Suppose that $\phi \leq d$ with equality at $p$ be a smooth viscosity barrier.   Given any $\epsilon > 0$, let   $h_{\epsilon} = d_{\epsilon} + \epsilon$ be as in the proof of Theorem \ref{t:visc}, so that 
\begin{align}
	\phi \leq d \leq h_{\epsilon} {\text{ with equality at }} p \, .
\end{align}
Since $p$ is a local minimum for $h_{\epsilon} - \phi$, the first derivative test and 
Theorem \ref{p:minimal} give at $p$
\begin{align}
	|\nabla_{\Gamma} \phi|^2 = |\nabla_{\Gamma} h_{\epsilon} |^2 = |\nabla_{\Gamma}  d_{\epsilon}|^2 \leq 1 - \epsilon_n \, .
\end{align}
This proves that  \eqr{e:whatn} holds in the viscosity sense as claimed. 

 For the second claim, Theorem \ref{t:visc} gives (in the viscosity sense) for any $p> 0$ that 
\begin{align}	\label{e:visc1A}
	\frac{d^{2-p}}{p} \,   \Delta_{\Gamma} \,  d^p  = d \, \Delta_{\Gamma} \, d + (p-1) \, |\nabla_{\Gamma} d|^2 
	  \leq  \left(  \frac{n-1}{n} \, |\nabla_{\Gamma} d|^2 + (p-1) \right)  \, |\nabla_{\Gamma} d|^2 \, .
\end{align}
Combining this with the first claim 
gives 
\begin{align}
	\frac{d^{2-p}}{p} \,   \Delta_{\Gamma} \,  d^p   
	  \leq  \left(  \frac{n-1}{n} \, (1 - \epsilon_n)  + (p-1) \right)  \, |\nabla_{\Gamma} d|^2 \, .
\end{align}
The right-hand side is nonpositive as long as $p \leq \frac{ 1 + (n-1) \, \epsilon_n}{n}$, completing the proof. 
\end{proof}

\section{Pushing apart locally}

In this section, $\Sigma^n \subset \RR^{n+1}$ is  a proper minimal immersion and $\Gamma_t^n \subset \RR^{n+1} \setminus \Sigma$ is a MCF.  We will prove Theorem \ref{t:apart} which shows that $\Gamma_t$ locally pushes away from $\Sigma$.

\begin{proof}[Proof of Theorem \ref{t:apart}]
We can assume that $x_0 = 0$. 
If $d(x,t)$ is the distance from $x \in \Gamma_t$ to $\Sigma$, then 
Theorem \ref{c:minimal} gives that $\Box_{\Gamma_t} \, d^{ \frac{1}{n}} \geq 0$ in the viscosity sense.  Thus, we see that $v = d^{ - \frac{1}{n}}$ satisfies
\begin{align}
	\Box_{\Gamma_t} \, v &= - \frac{ \Box_{\Gamma_t} d^{ \frac{1}{n}} }{ d^{ \frac{1}{n}}} - 2 \, \frac{ |\nabla_{\Gamma_t} d^{ \frac{1}{n}}|^2}{ \left( d^{ \frac{1}{n}} \right)^3} 
	= - \frac{ \Box_{\Gamma_t} d^{ \frac{1}{n}} }{ d^{ \frac{1}{n}}} - 2 \, \frac{|\nabla_{\Gamma_t} v|^2}{v}  \leq - 2 \, \frac{|\nabla_{\Gamma_t} v|^2}{v} \, . \label{e:forEH}
\end{align}
 As in the proof of Theorem $2.1$ in  \cite{EH}, define a cutoff function $\phi = (R^2 - |x|^2 - 2nt)_+$ and observe that (on the support of $\phi$)
 \begin{align}	\label{e:phit}
 	\Box_{\Gamma_t} \, \phi^2 \leq - 2 \, |\nabla_{\Gamma_t} \phi|^2 \, .
 \end{align}
 The product rule, \eqr{e:phit} and \eqr{e:forEH} give that
 \begin{align}	\label{e:forEH2}
 	\Box_{\Gamma_t} \, (\phi^2 \, v^2) &\leq - 6 \, \phi^2 \, |\nabla_{\Gamma_t} v|^2 - 8 \, \langle \phi \, \nabla_{\Gamma_t} v , v \, \nabla_{\Gamma_t} \phi \rangle
	- 2 \, v^2 \, |\nabla_{\Gamma_t} \phi |^2 \, .
 \end{align}
 If we note that
 \begin{align}
 	\langle \frac{\nabla_{\Gamma_t} \phi}{\phi} , \nabla_{\Gamma_t} (v^2 \, \phi^2) \rangle &= 2\, \left( v^2 \, |\nabla_{\Gamma_t} \phi|^2 + \phi \, v \, \langle
	\nabla_{\Gamma_t} \phi , \nabla_{\Gamma_t} v \rangle \right) \, , 
 \end{align}
 when we can rewrite \eqr{e:forEH2} as 
  \begin{align}	\label{e:forEH3}
 	\Box_{\Gamma_t} \, (\phi^2 \, v^2) &\leq 2\, \langle \frac{\nabla_{\Gamma_t} \phi}{\phi} , \nabla_{\Gamma_t} (v^2 \, \phi^2) \rangle
	- 6 \, \phi^2 \, |\nabla_{\Gamma_t} v|^2 - 12 \, \langle \phi \, \nabla_{\Gamma_t} v , v \, \nabla_{\Gamma_t} \phi \rangle
	- 6 \, v^2 \, |\nabla_{\Gamma_t} \phi |^2
	\notag \\
	&= 2\, \langle \frac{\nabla_{\Gamma_t} \phi}{\phi} , \nabla_{\Gamma_t} (v^2 \, \phi^2) \rangle - 6 \, 
	\left| \phi \, \nabla_{\Gamma_t} v - v \, \nabla_{\Gamma_t} \phi \right|^2 \leq 2\, \langle \frac{\nabla_{\Gamma_t} \phi}{\phi} , \nabla_{\Gamma_t} (v^2 \, \phi^2) \rangle
	 \, .
 \end{align}
 Therefore, we can apply the weak parabolic maximum principle to conclude that $\phi \, v$ has its maximum at the initial time. Thus, 
  we have for $t \geq t_0$ that
\begin{align}	\label{e:initialtime}
	d^{-1}  (x,t) \, \phi^n (x,t) \leq \sup_{\Gamma_{t_0}} \, \left( d^{ -1}  (x, t_0) \, \phi^n (x,t_0) \right) \, .
\end{align}
\end{proof}

\subsection{Pushing apart   for two flows}

We will next establish the  generalization Theorem \ref{t:twoflowslocal} of Theorem \ref{t:apart} that holds when both hypersurfaces are flows.  
Recall that, by Huisken's monotonicity theorem \cite{H},
 if $w\geq 0$ satisfies $\Box\,w\leq 0$, 
then
\begin{align}
	\frac{\partial}{\partial t} \, \int_{\Sigma_t} w\,\Phi \leq 0 \,  ,
\end{align}
where $\Phi (x,t)=(-4\,\pi\,t)^{-\frac{n}{2}}\,\exp\,\left(\frac{|x|^2}{4\,t}\right)$.    This gives a way of localizing estimates.   Theorem \ref{t:twoflowslocal} gives
a  stronger local estimate for the distance between two disjoint flows.

The next lemma is the key for the proof.
 
 \begin{Lem}	\label{l:exp}
If   $f= c_1 + c_2\, ( |x-x_0|^2+2\,n\,t)$, then in the viscosity sense
\begin{align}
	\Box_{\Gamma_t} \,\left( \e^f \,d^{-1} \right)\leq 0  \, .
\end{align}
\end{Lem}

\begin{proof}
We first show that 
if $\eta$ is a positive function and $p>0$, then  in the viscosity sense
\begin{align}		\label{e:introeta} 
\frac{\Box_{\Gamma_t} \,\left(\eta\,d^{-p}\right)}{\eta\,d^{-d}}\leq \Box_{\Gamma_t} \,\log \eta  
	-|\nabla_{\Gamma_t} \log\,(\eta\,d^{-p})|^2\,  .
\end{align}
Namely, 
  the chain rule gives that 
\begin{align}
\frac{\Box_{\Gamma_t} \,\e^{h}}{\e^h}=\Box_{\Gamma_t} \,h-|\nabla_{\Gamma_t}  h|^2\,  .
\end{align}
Applying this to $h=\log \, (\eta \, d^{-p})$   gives that
\begin{align}
	\frac{\Box_{\Gamma_t} \,\left(\eta\,d^{-p}\right)}{\eta\,d^{-d}} = \Box_{\Gamma_t} \, \log \eta - p \, \Box_{\Gamma_t}  \, \log d 
	-|\nabla_{\Gamma_t} \log\,(\eta\,d^{-p})|^2 \, .
\end{align}
The claim \eqr{e:introeta}  follows from this since $\Box_{\Gamma_t}  \log d \geq 0 $ by Theorem \ref{t:maind}.
 
Taking $\eta=\e^f$ and $p=1$, then  \eqr{e:introeta} gives that
\begin{align}	\label{e:generalef}
	(\eta\,d^{-1})^{-1}\,\Box_{\Gamma_t} \,(\eta\,d^{-1})
	&\leq \Box_{\Gamma_t} \,f\,  . 
\end{align}
 Suppose now that 
  $f= c_1 + c_2\, ( |x- x_0|^2+2\,n\,t)$, so that $\Box_{\Gamma_t} \,f=0$ on any MCF,   and  thus   \eqr{e:generalef} gives that 
\begin{align}
      \Box_{\Gamma_t} \,(\eta\,d^{-1}) \leq 0\,  .
\end{align}
\end{proof}

\begin{proof}[Proof of Theorem \ref{t:twoflowslocal}]
This follows from Lemma \ref{l:exp} and the weak maximum principle.  Note that the compactness of $\Gamma_t$ is used to apply the maximum principle. 
\end{proof}

\section{Two mean curvature flows}

There is a fully parabolic analog of the previous results, where both hypersurfaces flow by  mean curvature flow.
This fully parabolic estimate is not as strong as the previous one and there does not appear to be a corresponding
 improved gradient bound.  

Suppose now that $\Sigma_t \subset \RR^{n+1}$ is a properly immersed mean curvature flow. With the convention that we have used, this means  that
\begin{align}
	\partial_t \, x = H_{\Sigma_t} \, \nn_{\Sigma_t} \, .	\label{e:MCFw}
\end{align}
Let $d(x,t)$ be the distance from $x$ to $\Sigma_t$.  Given $t_0$, define
the set $\cD_0 (\Sigma_{t_0})$ as in Definition \ref{d:cD0}.

\begin{Lem}	\label{l:babyFn}
If $b^2 \leq 1$ and $c= \sqrt{1- b^2}$, then for $s \in [0,1)$
\begin{align}
	\frac{s^2}{1-s} - \frac{b^2\, s}{1-s} \geq - (1-c)^2 \, .
\end{align}
\end{Lem}

\begin{proof}
This follows by taking the limit as $\beta \to \infty$ in Lemma \ref{l:Fn}.
 \end{proof}

\begin{proof}[Proof of Theorem \ref{t:maind}]
Fix $t_0$ and a point $p_0 \in \Gamma_{t_0}$.  Let $\phi (x,t)$ be a parabolic viscosity barrier at $(p_0 , t_0)$, so that (near $(p_0,t_0)$)
\begin{align}
	\phi (x,t) \leq d(x,t) {\text{ for }}  t \leq t_0 {\text{ with equality at }} (p_0 , t_0) \, . \label{e:barrBOX}
\end{align}
Assume first that $p_0 \in \cD_0 (\Sigma_{t_0})$. Let $q_0 \in \Sigma_{t_0}$ be the unique closest point and $-1 < \alpha_1 \leq \alpha_2 \leq \dots$ the eigenvalues of $d\, A$ at $q_0$. 
Proposition \ref{p:mainset} gives at $(p_0 , t_0)$  that 
\begin{align}  \label{e:appDab}
	d\, \Delta_{\Gamma_t} \, d  - d \, H_{\Sigma_t}
	+ d\, H_{\Gamma_t} \, \langle \nn_{\Sigma_t}  , \nn_{\Gamma_t} \rangle
	+ \sum_{i=1}^n \frac{\alpha_i^2}{1+ \alpha_i}   \leq    - \frac{\alpha_1}{1+\alpha_1} \, |\nabla_{\Gamma_t} d|^2 
	 \, ,  \notag
\end{align}
where $\nn_{\Gamma_t}^T$ is the projection of $\nn_{\Gamma_t}$ perpendicular to $\nn_{\Sigma_t}$.  Since \eqr{e:barrBOX} implies that $d-\phi$ has a local minimum in space at $(p_0,t_0)$, the first and second derivative tests give that
$\nabla_{\Gamma_t} \phi (p_0 , t_0) = \nabla_{\Gamma_t} d (p_0 , t_0)$ and also at $(p_0, t_0)$
\begin{align}
	\phi \, \Delta_{\Gamma_t} \, \phi= d \, \Delta_{\Gamma_t} \, \phi \leq d\, \Delta_{\Gamma_t} \, d \, . 
\end{align}
Combining this with \eqr{e:appDab} gives at $(p_0 , t_0)$ that
\begin{align}	\label{e:gotttagoback}
	\phi \, \Delta_{\Gamma_t} \, \phi \leq  d \, H_{\Sigma_t}
	- d\, H_{\Gamma_t} \, \langle \nn_{\Sigma_t}  , \nn_{\Gamma_t} \rangle
	- \sum_{i=1}^n \frac{\alpha_i^2}{1+ \alpha_i}       - \frac{\alpha_1}{1+\alpha_1} \, |\nabla_{\Gamma_t} d|^2 
	 \, . 
\end{align}
Next,  let $p(t) \in \Gamma_t$ and $q(t) \in \Sigma_t$ be the flow lines through $p_0$ and $q_0$, respectively,
 with $p(t_0) = p_0$ and $q(t_0) = q_0$.  It follows that $d(p(t),t) \leq |p(t) - q(t)|$ with equality at $t=t_0$.  Combining this with \eqr{e:barrBOX} gives for $t \leq t_0$ that 
 \begin{align}
 	\phi (p(t), t) \leq  |p(t) - q(t)| {\text{ with equality at }} t=t_0 \, .
 \end{align}
 From this, we see that
 \begin{align}
 	\phi_t (p_0 ,t_0) &\geq \partial_t \big|_{t=t_0} \, |p(t) - q(t)| = \langle p_t (t_0) - q_t (t_0) , \frac{ p(t_0) - q(t_0)}{|p_0 - q_0|} \rangle
	\notag \\
	&=  \langle  H_{\Gamma_t} \, \nn_{\Gamma_t} - H_{\Sigma_t} \, \nn_{\Sigma_t}
	 ,\, -  \nn_{\Sigma_t} \rangle =  H_{\Sigma_t}  - H_{\Gamma_t} \,  \langle \nn_{\Gamma_t} , \nn_{\Sigma_t} \rangle \, ,
 \end{align}
 where the first equality on the second line used the mean curvature flow equation \eqr{e:MCFw}. Combining this with \eqr{e:gotttagoback}
 gives at $(p_0 , t_0)$ that
 \begin{align}
 	\phi \, \Box_{\Gamma_t} \, \phi \geq  \sum_{i=1}^n \frac{\alpha_i^2}{1+ \alpha_i}       + \frac{\alpha_1}{1+\alpha_1} \, |\nabla_{\Gamma_t} d|^2 =  \sum_{i=1}^n \frac{\alpha_i^2}{1+ \alpha_i}       + \frac{\alpha_1}{1+\alpha_1} \, |\nabla_{\Gamma_t} d |^2 
\, .
 \end{align}
 If $\alpha_1 \geq 0$, then this is nonnegative and we are done.  If $\alpha_1 \in (-1 , 0)$, then we 
 apply Lemma \ref{l:babyFn} with $b^2 =  |\nabla_{\Gamma_t} d |^2$ and 
 $s= - \alpha_1$ to get at $(p_0 , t_0)$ that
 \begin{align}
 	\phi \, \Box_{\Gamma_t} \, \phi \geq  - (1- \sqrt{1- b^2})^2 \geq - b^4 = -   |\nabla_{\Gamma_t} d |^4 = -  |\nabla_{\Gamma_t} \phi |^4
	\, .
 \end{align}
  This completes the proof when $p_0 \in  \cD_0 (\Sigma_{t_0})$.  In the remaining case, we replace $d$ by $d_{\epsilon} + \epsilon$ 
  and let $\epsilon \to 0$ as in the proof of Theorem \ref{t:visc}.
\end{proof}

\section{The distance to shrinking spheres}

We will next prove Proposition \ref{p:sphereX}.

\begin{proof}[Proof of Proposition \ref{p:sphereX}]
  If $\Gamma_t$ lies inside a sphere of radius $r(t)$, then $d(x)=r(t)-|x|$ and if it lies outside, then $d(x)=|x|-r(t)$.  
  Since $\Gamma_t$ is a MCF, $\Box_{\Gamma_t} \, |x|^2 = - 2\, n$ and
\begin{align}	\label{e:myspx}
\Box_{\Gamma_t} \,|x|=\frac{|\nabla_{\Gamma_t} \,|x||^2-n}{|x|}=\frac{|\nabla_{\Gamma_t} d|^2-n}{|x|}\,  .
\end{align}
Here the last equality used that $\nabla_{\Gamma_t} d=\pm \, \nabla_{\Gamma_t} |x|$.  Using \eqr{e:myspx} and
$r'(t)=-\frac{n}{r}$ gives  \eqr{e:D8}.

In the rest of this proof, $\Gamma_t$ will be inside of $\Sigma_t$.  Suppose that
 $x_0$ is a point  in $\Gamma_{t_0}$ so that  $x_0 \in T_{x_0} \Gamma_{t_0}$.  The last property gives at $(x_0 , t_0)$  that
   $|\nabla_{\Gamma_t} d| (x_0 , t_0) = 1$. 
 Set $d_0 = d(x_0 , t_0)$ and $r_0 = r(t_0)$. 
 Since $\Gamma_0$  lies inside the sphere $\Sigma_0$, we have at $(x_0 , t_0)$  that
\begin{align}
	|x_0| \, \Box_{\Gamma_t} \, d  (x_0 , t_0) = \frac{n\,d_0}{d_0+|x_0|}-|\nabla_{\Gamma_t} \,d|^2 (x_0 , t_0) = \frac{n \, d_0}{r_0} - 1 =
	\frac{ n \, d_0 - r_0}{r_0}
	\,  .
\end{align}
This gives the last claim.
 \end{proof}

\begin{Lem}	\label{l:cyl}
Suppose that $\Sigma_t = \SS^1_{\sqrt{-2t}} \times \RR^{n-1} \subset \RR^{n+1}$ and $\Gamma_t$ is a MCF inside it.
\end{Lem}

\begin{proof}
Set 
$s(x) = \sqrt{x_1^2 + x_2^2}$ and $r(t) =  \sqrt{-2\, t}$, so that $d(x,t) = r(t)- s(x(t))$ on $\Gamma_t$.  Since $\Box_{\Gamma_t} x_i = 0$, it follows that
\begin{align}
	 \Box_{\Gamma_t} \, d = - \frac{1}{r} + \frac{1}{s} \, \left( |\nabla_{\Gamma_t} s|^2 - |\nabla_{\Gamma_t} x_1|^2 
	 - |\nabla_{\Gamma_t} x_2|^2 \right) \, .
\end{align}
Suppose now that we are at a point where $|\nabla_{\Gamma_t} x_2|^2 = |\nabla_{\Gamma_t} x_1|^2 =1$, so that also $|\nabla_{\Gamma_t} s|^2 =1$, then we have 
\begin{align}
	 \Box_{\Gamma_t} \, d = - \frac{1}{r} - \frac{1}{s} {\text{ and }} |\nabla_{\Gamma_t} d| = 1  \, .
\end{align}
Using this and the chain rule, we get for $p> 0$  at this point that
\begin{align}
	\frac{d^{2-p}}{p} \,  \Box_{\Gamma_t} \, d^p &= d\,  \Box_{\Gamma_t} d+ (1-p) = 1-p
	-  (r-s) \, \left(\frac{1}{r} + \frac{1}{s} \right) 
\end{align}
\end{proof}

\appendix

 \section{Strong maximum principle}
 
 In this section,  $\Sigma_t , \Gamma_t \subset \RR^{n+1}$ are smooth  embedded proper connected hypersurfaces flowing by MCF for $t \in [t_1 , t_2]$ with $t_1 < t_2$ and $d(x,t)$ is the distance from $x \in \Gamma_t$ to $\Sigma_t$.   
  
 \begin{Thm}	\label{t:smp}
If there is an open set $\Omega \subset \RR^{n+1}$ and  a  point $p \in \Omega \cap \Gamma_{t_2}$, so that 
\begin{align}	\label{e:smp}
	d(p,t_2) = \inf_{t \in [t_1 , t_2]} \, \inf_{\Omega \cap \Gamma_{t}} d   \, ,
\end{align}
then $\Omega \cap \Gamma_t$ is a static   hyperplane and $\Sigma_t$ contains a translate of this.
 \end{Thm}
 
 We will use the following elementary fact in the proof:
 
 \begin{Lem}	\label{l:triangle}
 Suppose that $\Sigma$ is a proper  hypersurface in $\RR^{n+1}$,   $\dist (x , \Sigma) = r > 0$,  $z \in \RR^{n+1}$ has $|z| < r$, and $\Sigma' = \Sigma - z$ is the translate of $\Sigma$ by $z$.  Then 
 \begin{align}	\label{e:triangle}
 	\dist (x , \Sigma') \geq r - |z| 
 \end{align}
 and equality implies that $z$ is normal to $\Sigma'$ at any closest point to $x$.
 \end{Lem}
 
 \begin{proof}
 Let $y \in \Sigma'$ be a closest point to $x$, so that $|x-y| =  \dist (x , \Sigma')$ and $x-y$ is normal to $\Sigma'$ at $y$.
  Since $y+z \in \Sigma$, we have that
 \begin{align}
 	r = \dist (x , \Sigma) \leq |x- (y+z)| \leq |x - y| + |z| = \dist (x , \Sigma') + |z| \, ,
 \end{align}
 giving the inequality \eqr{e:triangle}.  If we have equality in \eqr{e:triangle}, then we conclude that $(y+z)$ is closest to $x$ in $\Sigma$ and $(x - y)$ and $z$ are parallel.  Since $x-y$ is normal to $\Sigma'$ at $y$, we see that $z$ is also normal there.
 \end{proof}
 
 \begin{proof}[Proof of Theorem \ref{t:smp}]
By Theorem \ref{t:maind},  $\Box_{\Gamma_t} \, \log d  \geq 0 $
in the viscosity sense.  Therefore, \eqr{e:smp} and the  strong maximum principle  give  that  $d\equiv  r_0 > 0$    is constant on $\Omega \cap \Gamma_t$ for $t \in [t_1 , t_2]$.   In particular, given any $t$ and $x \in \Omega \cap \Gamma_t$, there is a closest point $y_x$ in $\Sigma_t$ with 
\begin{align}
	|x-y_x| = r_0 > 0   {\text{ and }} (y_x-x) = \pm r_0 \, \nn (y_x) \, .
\end{align}

Let $y_p$ in $\Sigma_{t_2}$ be the closest point to $p$ and set $z_p = y_p - p$.  Note that $|z_p| = r_0$ by construction. 
Given some $s \in (0,1)$, 
define the translated hypersurface $\Sigma_t' = \Sigma_t - s\, z_p$.  Note that this is also a    mean curvature flow.
Thus, if we let $d' (x,t)$ be the distance from $x$ to $\Sigma_t'$, then 
Theorem \ref{t:maind} gives that  $\Box_{\Gamma_t} \, \log d'  \geq 0 $.  Furthermore, Lemma \ref{l:triangle} gives that 
\begin{align}	\label{e:dtriangle}
	d'(x,t) \geq d(x,t) - s \, |z_p| = (1-s)\, r_0 \, .
\end{align}
Moreover, by construction, we get equality at $(p,t_2)$, so the strong maximum principle gives that $d' \equiv (1-s) \, r_0$ is constant for $t\leq t_2$. 
In particular,  for $t\leq t_2$, 
we have equality in \eqr{e:dtriangle}  and Lemma \ref{l:triangle} gives for each $x \in \Gamma_t$ that there is a closest point $y_x \in \Sigma_t'$ which has $z_p$ as a normal vector at $y_x$. Since every point in $\Omega \cap \Gamma_t$ is at least as far away from $y_x$, the first derivative test gives that $y_x -x$ is normal also at $x$.  Thus, since $y_x - x$ is parallel to $z_p$ (by the lemma), we see that  the (constant vector)  $\frac{z_p}{|z_p|}$ is a unit normal at every $x \in \Omega \cap \Gamma_t$.  This says that $\Omega \cap \Gamma_t$ is a hyperplane.   
    \end{proof}

\end{document}